\documentclass[11pt]{amsart}

\usepackage{amscd,amssymb,amsopn,amsmath,amsthm,mathrsfs,graphics,amsfonts,enumerate,verbatim,calc}

\usepackage{amssymb,amsmath, tikz,xypic}
\input xy

\DeclareFontFamily{OT1}{rsfs}{}
\DeclareFontShape{OT1}{rsfs}{n}{it}{<-> rsfs10}{}
\DeclareMathAlphabet{\mathscr}{OT1}{rsfs}{n}{it}

\numberwithin{equation}{section}
\newtheorem{theorem}{Theorem}[section]
\newtheorem{lemma}[theorem]{Lemma}

\newtheorem{corollary}[theorem]{Corollary}

\newtheorem*{maintheorem}{Main Theorem}

\theoremstyle{definition}
\newtheorem{definition}[theorem]{Definition}
\newtheorem{remark}[theorem]{Remark}
\newtheorem{example}[theorem]{Example}

\newcommand{\im}{\operatorname{Im}}
\renewcommand{\ker}{\operatorname{ker}}

\newcommand{\Spec}{\operatorname{Spec}}

\newcommand{\Ext}{\operatorname{Ext}}

\newcommand{\Hom}{\operatorname{Hom}}

\newcommand{\BN}{\mathbf N}

\newcommand{\bD}{\mathbf D}

\newcommand{\fm}{\frak{m}}
\newcommand{\fp}{\frak{p}}
\newcommand{\fq}{\frak{q}}

\newcommand{\fn}{\frak{n}}

\newcommand{\RGamma}{\mathbf{R}\Gamma}

\newcommand{\cA}{\mathcal{A}}

\newcommand{\RMod}{R\text{--mod}}

\def\PSpec{\operatorname{Spec}^\circ\!}

\makeatletter
\let\@wraptoccontribs\wraptoccontribs
\makeatother

\begin{document}
\title[Deformation of $F$-injectivity and local cohomology]
{Deformation of $F$-injectivity and local cohomology}

\author[J. Horiuchi]{Jun Horiuchi}
\address{Department of Mathematics, Nippon Institute of Technology, Miyashiro, Saitama
345-8501, Japan}
\email{jhoriuchi.math@gmail.com}
 
\author[L.E. Miller]{Lance Edward Miller}
\address{Department of Mathematics, Department of Mathematics, 301 SCEN, University of Arkansas, Fayetteville, AR~72701,~USA}
\email{lem016@uark.edu}
\thanks{The second author was supported in part by National Science Foundation VIGRE Grant \# 0602219.}

\author[K. Shimomoto]{Kazuma Shimomoto}
\address{Department of Mathematics, School of Science and Technology, Meiji University, 1-1-1 Higashimita, Tama-Ku, Kawasaki 214-8571, Japan}
\email{shimomotokazuma@gmail.com}
\thanks{The third author was supported in part by Grant-in-Aid for Research Activity Start-up Grant \# 22840042}

\contrib[with an appendix by]{Karl Schwede}
\address{Department of Mathematics, The Pennsylvania State University, University Park, PA~$16802$, USA}
\email{schwede@math.psu.edu}

\contrib[]{Anurag K. Singh}
\address{Department of Mathematics, University of Utah, 155 S 1400 E, Salt Lake City, UT~84112,~USA}
\email{singh@math.utah.edu}

\thanks{The fourth author was supported by NSF grant~\# 1064485 and a Sloan Research Fellowship.} \thanks{The fifth author was supported by NSF grant DMS~ \# 1162585.}

%




\begin{abstract} 
We give a sufficient condition for $F$-injectivity to deform in terms of local cohomology. We show this condition is met in two geometrically interesting setting, namely when the special fiber has isolated non-CM locus or is $F$-split. 
\end{abstract}

\maketitle

\section{Introduction}

A central and interesting question in the study of singularities is how they behave under deformation. Given a local ring of positive characteristic, view this ring as the total space of a fibration. The special fiber of this fibration is a hypersurface in $R$, i.e., a variety with coordinate ring $R/xR$ where $x \in R$ is a regular element. An important question is whether or not the singularity type of the total space $R$ is no worse than the singularity type as the special fiber. This deformation question has been studied in detail for singularities defined by Frobenius \cite{Fed,Singhb} where it is noted that $F$-rationality always deforms and both $F$-purity and $F$-regularity fail to deform in general. An important and outstanding conjecture asserts that $F$-injectivity deforms in general. Recall that a local ring $(R, \fm)$ of prime characteristic $p>0$ is \textit{F-injective} provided the Frobenius action on the local cohomology $H^i_{\fm}(R)$ induced by the Frobenius map on $R$, is injective for all $i \ge 0$. The general conjecture is supported by recent work showing the characteristic $0$ analogue of $F$-injective singularities, called Du Bois singularities, deform \cite{KoSch1}.  When $R$ is Cohen-Macaualy, it is known that $F$-injectivity deforms \cite{Fed}. Our main theorem describes a condition sufficient to guarantee $F$-injective to deform which only requires information about the special fiber and not the total space.

\begin{maintheorem}
$($c.f., Theorem \ref{thm:MainTheorem}$)$
Let $(R,\fm,k)$ be a local ring of prime characteristic $p>0$ and $x \in \fm$ a regular element. If $R/xR$ is $F$-injective and for each $\ell > 0$ and $i \geq 0$ the homomrphism $H_\fm^i(R/x^\ell R) \to H_\fm^i(R/xR)$ induced by the natural surjection $R/x^\ell R \to R/xR$, is surjective, then $R$ is $F$-injective. 
\end{maintheorem}

We show in particular that this hypothesis is satisfied when the length of the local cohomology modules $H_\fm^i(R/xR)$ is finite for $i < \dim R-1$; a condition called {\it finite length cohomology} (FLC). Geometrically, it is the condition that the non Cohen-Macaualy locus of $R/xR$ is isolated and this combination shows that $F$-injectivity deforms under mild geometric criteria in low dimensions, see Corollary \ref{cor:FinjGeomlowdim}.

\begin{maintheorem}
$($c.f., Corollary \ref{cor:FLCFINJ}$)$
Let $(R,\fm,k)$ be a local ring of characteristic $p>0$ with perfect residue field and $x \in \fm$ a regular element. If $R/xR$ has FLC and is $F$-injective, then $R$ is $F$-injective. 
\end{maintheorem}

Utilizing a sharper study of Frobenius actions on local cohomology, we can state our condition in terms of the condition \textit{anti-nilpotency}. Using results of L. Ma, we demonstrate a deformation theoretic relationship between $F$-injectivity and $F$-splitting.

\begin{maintheorem}
$($Corollary \ref{cor:FsplitFinj}$)$
 Let $(R,\fm,k)$ be a local ring of characteristic $p > 0$ and $x \in \fm$ a regular element. If $R/xR$ is $F$-split, then $R$ is $F$-injective. 
\end{maintheorem}

The authors thank Takeshi Kawasaki, Karl Schwede, and Anurag Singh for many helpful discussions and careful readings of this manuscript as well as Alberto F. Boix  for suggestions leading to improvements in the manuscript and an alternate proof of Lemma~\ref{lem:SS}. Special thanks also goes to Linquan Ma for finding an error in a previous version of Theorem~\ref{thm:surjRspan}. We also thank the referee for comments and suggestions leading to further improvement. 

\

\noindent {\it Conventions:} Unless otherwise stated all rings are noetherian and of characteristic $p > 0$ where $p$ is a prime integer.

\section{Preliminaries and Notations}

\subsection{Notation}
For a ring $R$ of characteristic $p > 0$, the Frobenius map is the map $F \colon R \to R$ sending an element to its $p$-th power. For an $R$-module $M$, denote by $F_*M = \{ F_*m \colon m \in M\}$. This module is called the \textit{Frobenius pushforward} of $M$. As abelian groups $M \cong F_*M$, but its $R$-module structure is twisted by Frobenius. In particular, if  $r \in R$ and $F_*m \in F_*M$, then $r \cdot F_*m= F_* (r^pm)$. We also denote the $e$-th iterate of the Frobenius pushforward of $M$ by $F_*^eM$. The functor $F_*^e$ is exact and commutes with localization.

\subsection{Local cohomology}

For a more complete introduction see \cite{ILL}. Fix a ring $R$ and an ideal $I$. Let $M$ be an $R$-module; not necessarily noetherian. The local cohomology module supported at $I$ is $H_I^i(M) = \varinjlim_{t} \Ext_R^i(R/I^t,M)$. When $I$ is generated up to radical by $g_1,\ldots,g_n$, one may compute $H_I^i(M)$ as the $i$-th cohomology of the \v{C}ech complex with respect to $I$, denoted $\check{C}^\bullet(M;I)$; 
$$0 \to M \to \oplus_i M_{g_i} \to \oplus_{i < j} M_{g_ig_j} \to \cdots \to M_{g_1\cdots g_n} \to 0.$$

We briefly discuss iterated local cohomology as it plays a role in the proof of Theorem \ref{thm:MainTheorem}. For more detail see \cite{Har67}. Given two ideals $I$ and $J$ in $R$, and an $R$-module $M$, let $\check{C}^\bullet (M; I)$ (resp. $\check{C}^\bullet(M;J)$) be the \v{C}ech complex of $M$ with respect to $I$ (resp. with respect to $J$). Considering $\check{C}^\bullet(M;I)$ as the horizontal complex and $\check{C}^\bullet(M;J)$ as the vertical complex, one obtains a double complex $C^{\bullet \bullet} = \check{C}^\bullet(M;I) \otimes_R \check{C}^\bullet(M;J)$. This double complex is the first page of a spectral sequence $E_0^{p,q}$, called the {\it local cohomology spectral sequence}. For more on spectral sequences see \cite{Wei94}. The convergence of this spectral sequence is known.

\begin{theorem}
$($Convergence of local cohomology spectral sequence \cite[Prop. 1.4]{Har67}$)$
\label{thm:LCSS}
For $I$ and $J$ ideals in a ring $R$ and $M$ an $R$-module, 
$$
E_2^{p,q} = H_J^p(H_I^q(M)) \Rightarrow E_\infty^{p,q}  = H_{I + J}^{p+q}(M).
$$
\end{theorem}

Using this theorem, it is easy to compute an isomorphism that we need.

\begin{lemma}
\label{lem:SS}
Let $(R,\fm,k)$ be a local ring and $M$ an $R$-module. If $x \in \fm$ is a regular element, then for all $i \geq 0$, $$H_\fm^i( H_{(x)}^1(M)) \cong H_\fm^{i+1}(M).$$ 
\end{lemma}

\begin{proof}
First, note that $H_{(x)}^q(M)$ is nonzero only when $q = 1$. Thus the $E_2^{p,q}$ page of the spectral sequence computing the double complex $H_\fm^\bullet(H_{(x)}^\bullet(-))$ degenerates. By Theorem~\ref{thm:LCSS}, $E_2^{p,q} = H_\fm^p(H_{(x)}^q(M))$ and $E_\infty^{p,q} = H_\fm^{p + q}(M)$ for all $p \geq 0$ and $q \geq 0$. Since the sequence degenerates at the $E_2^{p,q}$ page, we have $H_\fm^p(H_{(x)}^q(M)) = E_2^{p,q} = E_\infty^{p,q} = H_\fm^{p + q}(M)$ for all $p \geq 0$ and $q \geq 0$. Applying this with $p = i$ and $ q = 1$ gives the result. 
\end{proof}

We also offer a second proof of Lemma~\ref{lem:SS} free of spectral sequences due to Alberto Boix. 

\begin{proof}(proof of Lemma~\ref{lem:SS} with thanks to Alberto Boix)
Note the following exact sequence of $R$-modules $$0 \to \Gamma_{(x)}(M) \to M \to M_{x} \to H^1_{(x)}(M) \to 0.$$ Since $x$ is not a zero divisor, $\Gamma_{(x)}(M) = 0$ so we have a short exact sequence $$0 \to M \to M_{x} \to H_{(x)}^1(M) \to 0.$$ Taking $H^i_\fm$ induces a long exact sequence $$H_\fm^i(M_{x}) \to H_\fm^i(H_{x}^1(M))\to H_\fm^{i+1}(M) \to H_\fm^{i+1}(M_x).$$ One may then apply flat base change to check for all $j \geq 0$ that $H_\fm^j(M_x) \cong H_{\fm R_x}^j(M_x)$ and since $x \in \fm$ the extension $\fm R_x = R_x$. Thus $H_\fm^j(M_x) \cong H_{R_x}^j(M_x) = 0$. This gives the desired result. 
\end{proof}

\begin{remark}
Neither proof of Lemma~\ref{lem:SS} depends on the local cohomology being supported in the maximal ideal, only that the regular element $x$ is a member of the ideal of support for the local cohomology modules in question. 
\end{remark}

It is often easier to study spectral sequences as composition of derived functors; see \cite{Lip02} for explicit details about derived categories and local cohomology. We summarize what we need. For an abelian category $\cA$, denote by $K(\cA)$ the category of complexes in $\cA$ up to homotopic equivalence and $\bD(\cA)$ its derived category. For $R$ a ring, denote by $\RMod$ the category of $R$-modules. Let $I \subseteq R$ an ideal and $\cA = \RMod$. One realizes the $i$-th local cohomology module with support in $I$ as a functor $H_I^i \colon K(\RMod) \to \RMod$ which takes quasi-isomorphisms in $K(\RMod)$ to isomorphisms in $\RMod$ and so it can be regarded as a functor on $\bD(\RMod)$. Denote by $\Gamma_I$ the $I$-torsion functor. The right derived functor $\RGamma_I \colon \bD(\RMod) \to \bD(\RMod)$ has the information of taking all of the local cohomology modules $H_I^i$ at once and each $H_I^i$ can be recovered in a functorial way from $\bD(\RMod)$ by taking the $i$-th cohomology of the image of $\RGamma_I$. The spectral sequence in Theorem \ref{thm:LCSS} can be understood as a consequence of the Grothendieck spectral sequence theorem \cite[Cor. 10.8.3]{Wei94} stating that $\RGamma_I \circ \RGamma_J \cong \RGamma_{I + J}$. This equivalence will be utilized in Theorem \ref{thm:MainTheorem}.

\subsection{Frobenius linear maps}

Frobenius linear maps are a central tool in our approach. These are thoroughly explored in \cite{HarSp} under the name $p$-linear maps. We review the topic.

\begin{definition}
Let $R$ be a commutative ring of characteristic $p$. For $R$-modules $M$ and $N$, a {\it Frobenius linear map} is an element of $\Hom_R(M,F_* N)$. More specifically, it is an additive map $\rho \colon M \to F_*M$  such that $\rho(ra)=r^p\rho(a)$ for any $r \in R$ and $a \in M$. If $M=N$, we call $\rho \colon M \to F_* M$ a \textit{Frobenius action} on $M$. 
\end{definition}

Since $F_*$ commutes with localization, given a Frobenius linear map between $M$ and $N$ there is an induced Frobenius linear map $H_\fm^i(M) \to F_* H_\fm^i(N)$ for each $i \geq 0$. One can make this explicit utilizing \v{C}ech resolutions as in Example \ref{xmp:Cech}. Any Frobenius linear map $\rho \colon M \to F_* N$ induces a morphism $\RGamma_I( \rho) \colon \RGamma_I (M) \to \RGamma_I (F_* N) \cong F_* \RGamma_I (N)$ where $I \subseteq R$ is an ideal and the last isomorphism follows as $F_*$ is exact. In particular, the Frobenius map on $R$, thought of as a Frobenius action $\rho_F \colon R \to F_* R$, induces a natural Frobenius action on the local cohomology 
$$
\RGamma_I(\rho_R) \colon \RGamma_I(R) \to F_* \RGamma_I(R).
$$
This Frobenius action can be computed explicitly using \v{C}ech complexes.

\begin{example}
\label{xmp:Cech}
Consider $(R,\fm,k)$ a local ring with $x \in \fm$ a regular element. Each term of the \v{C}ech complex 
$$
0 \to R \to R_x \to 0
$$ 
has a Frobenius linear map induced from the Frobenius on $R$. Therefore we have a commutative diagram 

\begin{center}
\begin{tikzpicture}[baseline=(current bounding box.center)]

\node (RT0) at (-4,1) {$0$};
\node (RT1) at (-1.5,1){$R$};
\node (RT2) at (1.5,1){$R_x$};
\node (RT3) at (4,1){$0$};

\node (RB0) at (-4,-1){$0$};
\node (RB1) at (-1.5,-1){$F_*R$};
\node (RB2) at (1.5,-1){$F_*(R_x)$};
\node (RB3) at (4,-1) {$0$};

\draw[->] (RT0) edge  (RB0);
\draw[->] (RT1) edge node[left]{$\rho_{F}$} (RB1);
\draw[->] (RT2) edge node[left]{$\rho_{F}$} (RB2);
\draw[->] (RT3) edge  (RB3);

\draw[->] (RT0) edge (RT1);
\draw[->] (RT1) edge (RT2);
\draw[->] (RT2) edge (RT3);

\draw[->] (RB0) edge (RB1);
\draw[->] (RB1) edge (RB2);
\draw[->] (RB2) edge (RB3);

\end{tikzpicture}
\end{center}

Of course $H_{(x)}^0(R) = 0$ and $H_{(x)}^1(R) = R_x/R$. Since $F_*$ commutes with localization, it also commutes with local cohomology. Therefore, we have a natural Frobenius action $\rho$ on the $R$-module $H_{(x)}^1(R) = R_x/R$. In particular, $\rho \colon H_{(x)}^i(R) \to F_* H_{(x)}^i(R)$ is just the natural Frobenius $R_x/R \to F_*(R_x/R)$. 
\end{example}

We see immediately the benefit of studying Frobenius linear maps on finite length modules when the residue field is perfect.

\begin{lemma}
\label{lemma2}
Let $(R,\fm,k)$ be a local ring of prime characteristic $p>0$ with perfect residue field and let $M$ be an $R$-module  admitting an injective Frobenius action $\rho$. If $M$ has finite length, then $M$ is a finite dimensional $k$-vector space and $\rho$ is a bijection.
\end{lemma}

\begin{proof}
Since $M$ has finite length, there exists $\ell>0$ such that $\fm^{\ell} \cdot M=0$. Fix $c \in \fm$. One has $\rho^e(c \cdot M)=c^{p^e} \cdot \rho(M)=0$ for $p^e \ge \ell$. Since $\rho$ is injective, $c \cdot M=0$. Therefore, $M$ is a finite dimensional $k$-vector space and $\rho$ descends to an additive map on $M=M/\fm M$. Now since $k$ is perfect and $M$ is finite dimensional as $M$ has finite length and $\rho$ is injective, $\rho$ must be bijective. 
\end{proof}

\begin{remark} The perfectness of the residue field in Lemma \ref{lemma2} is necessary. In the case, $R = k$, the natural Frobenius action on the simple $k$-module $k$ is bijective if and only if $k$ is perfect. See also \cite[Cor. 7.7 and Prop. 7.12]{Ene} for similar discussion.
\end{remark}





\section{Proof of the main theorem}

We start with the following notation defining the key property about a regular element that we need to guarantee that $F$-injective deforms.

\begin{definition}
Let $(R,\fm)$ be a local ring with $x \in \fm $ a regular element. We say that $x$ is a \textit{surjective element}, if the map on local cohomology $H^i_{\fm}(R/x^{\ell}R) \to H^i_{\fm}(R/xR)$, which is induced by the natural surjection $R/x^\ell R \to R/xR$, is surjective for all $\ell>0$ and $i \ge 0$.
\end{definition}

We immediately see that surjective elements induce {\it injections} between specific local cohomology modules. 

\begin{lemma}
\label{cor:inj}
Let $(R,\fm)$ be a local ring of arbitrary characteristic. Assume that $x \in \fm$ is a surjective element. For each $\ell > 0$ and $j \geq \ell$, the multiplication map 
$$
R/x^\ell R \xrightarrow{x^{j - \ell} } R/x^jR
$$ 
induces an injection $H_\fm^i(R/x^\ell R) \to H_\fm^i(R/x^jR)$ for each $i \ge 0$. 
\end{lemma}

\begin{proof}
Note that $R/x^\ell R \xrightarrow{x^{j - \ell} } R/x^jR$ is injective and it suffices by induction to prove the lemma when $j=\ell+1$. The short exact sequence $$0 \to R/x^{\ell} R \xrightarrow{\cdot x} R/x^{\ell + 1} R \to R/xR \to 0$$ induces the following exact portion of the long exact sequence  
$$
H_\fm^{i-1}(R/x^{\ell+1} R) \xrightarrow{\beta_1} H_\fm^{i-1}(R/xR) \xrightarrow{\delta} H_\fm^i(R/x^{\ell } R) \xrightarrow{\beta_2} H_\fm^i(R/x^{\ell + 1}R).
$$ 
Since $x$ is a surjective element, $\beta_1$ is surjective and hence $\delta$ is the zero map. This makes $\beta_2$ injective as desired.
\end{proof}

\begin{theorem}
\label{theorem1}
Let $(R,\fm,k)$ be a local ring of characteristic $p > 0$ and let $x \in \fm$ be a surjective element. Assume that $R/xR$ is $F$-injective and denote by 
$$
\rho_{\ell,i} \colon H_\fm^i(R/x^\ell R) \to F_* H_\fm^i(R/x^{p \ell}R),
$$ 
the Frobenius linear map induced by the natural Frobenius map $\rho_F:R/x^\ell R \to F_*(R/x^{p \ell} R)$. For each $\ell > 0$ and $i \ge 0$, the map $\rho_{\ell,i}$ is injective.
\end{theorem}

\begin{proof}
For every $\ell > 0$, the natural Frobenius map on $R/x^\ell R$ is a composition of $\rho_F$ and a natural surjection $\pi$, i.e., 
$$
R/x^\ell R \xrightarrow{\rho_F} F_*(R/ x^{p \ell } R) \xrightarrow{\pi} F_*(R/ x^\ell R).
$$ 
Denote by $\rho_{\ell, i} \colon H_\fm^i(R/x^\ell R) \to F_* H_\fm^i(R/x^{p \ell } R)$ the Frobenius linear map induced by $\rho_F$. We proceed by induction on $\ell$ to show that $\rho_{\ell,i}$ is injective for all $\ell > 0$. The case $\ell = 1$ is assured by hypothesis.

Assume $\ell >1$ and consider the commutative diagram of $R$-modules with exact rows
\begin{equation}\label{eq:diag1}
\begin{tikzpicture}[baseline=(current bounding box.center)]

\node (RT0) at (-10,1) {$0$};
\node (RT1) at (-7,1){$R/x^{\ell - 1}R$};
\node (RT2) at (-3,1){$R/x^\ell R$};
\node (RT3) at (1,1){$R/xR$};
\node (RT4) at (3,1){$0$};

\node (RB0) at (-10,-1){$0$};
\node (RB1) at (-7,-1){$F_*(R/x^{p(\ell -1)}R)$};
\node (RB2) at (-3,-1){$F_*(R/x^{p \ell}R)  $};
\node (RB3) at (1,-1) {$F_*(R/x^pR) $};
\node (RB4) at (3,-1) {$0 $};

\draw[->] (RT1) edge  (RB1);
\draw[->] (RT2) edge  (RB2);
\draw[->] (RT3) edge  (RB3);

\draw[->] (RT0) edge (RT1);
\draw[->] (RT1) edge  node[above]{$ \cdot x $}(RT2);
\draw[->] (RT2) edge (RT3);
\draw[->] (RT3) edge (RT4);

\draw[->] (RB0) edge (RB1);
\draw[->] (RB1) edge  node[above]{$ \cdot x^p $}(RB2);
\draw[->] (RB2) edge (RB3);
\draw[->] (RB3) edge (RB4);

\end{tikzpicture}
\end{equation}
where all vertical maps are the natural Frobenius linear maps. This induces the following commutative diagram of $R$-modules

{ \footnotesize{
\begin{equation}\label{eq:diag2}
\begin{tikzpicture}[baseline=(current bounding box.center)]

\node (RT0) at (-11,1) {$H^{i-1}_{\fm}(R/xR)$};
\node (RT1) at (-7,1){$H^i_{\fm}(R/x^{\ell-1}R)$};
\node (RT2) at (-3,1){$H^i_{\fm}(R/x^{\ell}R)$};
\node (RT3) at (1,1){$H^i_{\fm}(R/xR)$};

\node (RB0) at (-11,-1){$F_*H^{i-1}_{\fm}(R/x^pR)$};
\node (RB1) at (-7,-1){$F_*H^i_{\fm}(R/x^{p(\ell-1)}R) $};
\node (RB2) at (-3,-1){$F_*H^i_{\fm}(R/x^{p\ell}R)  $};
\node (RB3) at (1,-1) {$F_*H^i_{\fm}(R/x^pR) $};

\draw[->] (RT0) edge node[left]{$\rho_{1, i-1}$} (RB0);
\draw[->] (RT1) edge node[left]{$\rho_{\ell-1, i}$} (RB1);
\draw[->] (RT2) edge node[left]{$\rho_{\ell, i}$} (RB2);
\draw[->] (RT3) edge node[left]{$\rho_{1, i}$} (RB3);

\draw[->] (RT0) edge (RT1);
\draw[->] (RT1) edge (RT2);
\draw[->] (RT2) edge node[above]{$\alpha$}(RT3);

\draw[->] (RB0) edge node[above]{$F_* \delta_{i-1}$}(RB1);
\draw[->] (RB1) edge node[above]{$F_* \beta$}(RB2);
\draw[->] (RB2) edge (RB3);

\end{tikzpicture}
\end{equation}
}}

The map $\alpha:H^i_{\fm}(R/x^{\ell}R) \to H^i_{\fm}(R/xR)$ is surjective, since $x$ is a surjective element by assumption. From Lemma \ref{cor:inj} and that $F_*$ is exact, the map $F_* \beta$ is injective. Hence $F_* \delta_{i-1}$ is the zero map. Thus we have a commutative diagram

{ \footnotesize{
\begin{equation}\label{eq:diag3}
\begin{tikzpicture}[baseline=(current bounding box.center)]

\node (RT1) at (-7,1){$H^i_{\fm}(R/x^{\ell-1}R)$};
\node (RT2) at (-3,1){$H^i_{\fm}(R/x^{\ell}R) $};
\node (RT3) at (1,1){$H^i_{\fm}(R/xR)$};
\node (Z1) at (4,1) {$0$};

\node (Z2) at (-10,-1) {$0$};
\node (RB1) at (-7,-1){$F_*H^i_{\fm}(R/x^{p(\ell-1)}R)$};
\node (RB2) at (-3,-1){$F_*H^i_{\fm}(R/x^{p\ell}R) $};
\node (RB3) at (1,-1){$F_*H^i_{\fm}(R/x^pR)  $};

\draw[->] (RT1) edge node[left]{$\rho_{\ell-1, i}$} (RB1);
\draw[->] (RT2) edge node[left]{$\rho_{\ell, i}$} (RB2);
\draw[->] (RT3) edge node[left]{$\rho_{1, i}$} (RB3);

\draw[->] (RT1) edge (RT2);
\draw[->] (RT2) edge (RT3);
\draw[->] (RT3) edge (Z1);

\draw[->] (Z2) edge (RB1);
\draw[->] (RB1) edge (RB2);
\draw[->] (RB2) edge (RB3);

\end{tikzpicture}
\end{equation}
}}

To complete the argument, apply the snake lemma to Diagram~\eqref{eq:diag3}. This gives an exact sequence $\ker \rho_{\ell-1,i} \to \ker \rho_{\ell,i} \to \ker \rho_{1,i}$. Since $\rho_{1,i}$ is injective by $F$-injectivity of $R/xR$, and $\rho_{\ell -1, i}$ is injective by induction, we have that $\ker \rho_{\ell,i} = 0$. Hence $\rho_{\ell,i}$ is injective. 
\end{proof}

\begin{remark}
The specific point where $x$ being a surjective element was used was to obtain that $F_*\beta$ in Diagram~\eqref{eq:diag2} is injective. The fact that $\alpha$ is surjective is not really required as one can do a straightforward chase on Diagram~\eqref{eq:diag3} similar to how one proves the snake lemma to conclude the result. 
\end{remark}

We record a lemma used in the proof of the main theorem whose proof is left to the reader.

\begin{lemma}
\label{lem:Finjlim}
For a directed system $\{ N_i, \tau_{i,j} \}_{i \in \Lambda}$ of $R$-modules, the system $\{ F_* N_i, F_* \tau_{i,j} \}_{i \in \Lambda}$ is also directed and $F_* \varinjlim N_i \cong \varinjlim F_* N_i$. 
\end{lemma}

The next lemma explains the basic isomorphisms needed in the proof of the main theorem. 

\begin{lemma}
\label{lem:isos}
Let $(R,\fm)$ be a local ring with $x \in \fm$ a regular element. For each $i > 0$, we have isomorphisms
$$
H_\fm^i(H_{(x)}^1(R)) \cong H_{\fm}^{i+1}(R) \cong \varinjlim_{\ell} H_{\fm}^i(R/x^{\ell} R) \cong \varinjlim_{\ell} H_\fm^i(R/x^{p\ell}R).
$$
\end{lemma}

\begin{proof}
We show this by showing that $R$-modules $H_\fm^{i+1}(R)$, $\varinjlim_{\ell} H_{\fm}^i(R/x^{\ell} R)$, and $\varinjlim_{\ell} H_\fm^i(R/x^{p\ell}R)$ are all isomorphic to the iterated local cohomology module $H_\fm^i(H_{(x)}^1(R))$. Computing $H^1_{(x)}(R)$ as  
$$ 
\varinjlim \{R/xR \xrightarrow{x} R/x^2R \xrightarrow{x} R/x^3R \xrightarrow{x} \cdots \},
$$
and noting that local cohomology commutes with direct limits, one has 
$$
\varinjlim_{\ell} H^i_{\fm}(R/x^{\ell}R) \cong H^i_{\fm}(\varinjlim_{\ell} R/x^{\ell}R) \cong H^i_{\fm}(H^1_{(x)}(R)).
$$
By Lemma \ref{lem:SS}, 
$$
\varinjlim_{\ell} H^i_{\fm}(R/x^{\ell}R) \cong H^i_{\fm}(H^1_{(x)}(R)) \cong H^{i+1}_{\fm}(R).
$$ 
Since $\{x^{p\ell}\}_{\ell \in \BN}$ is cofinal in $\{x^\ell\}_{\ell \in \BN}$, one can compute $H^1_{(x)}(R)$ as the limit 
$$ \varinjlim \{R/x^pR \xrightarrow{x^p} R/x^{2p}R \xrightarrow{x^p} R/x^{3p}R \xrightarrow{x^p} \cdots \},$$
and like before we have
$\varinjlim_{\ell} H^i_{\fm}(R/x^{p\ell}R) \cong H^i_{\fm}(H^1_{(x)}(R)).$
\end{proof}

We now prove the main theorem of this article.

\begin{theorem}
\label{thm:MainTheorem}
Let $(R,\fm,k)$ be a local ring of prime characteristic $p>0$ and $x \in \fm$ is a regular surjective element. If $R/xR$ is $F$-injective, then $R$ is also $F$-injective. 
\end{theorem}

\begin{proof}

Since $R$ has a regular element $x$, $H^0_{\fm}(R)=0$ and there is nothing to prove in the case $i=0$. Fix $ i > 0$ and consider the following commutative diagram of $R$-modules, where $\rho_F$ denotes the natural Frobenius map

\begin{equation}\label{eq:MainDiagLim}
\begin{tikzpicture}[baseline=(current bounding box.center)]
\pgfmathsetmacro{\L}{2}
\node (NW) at (-\L,\L) {$R/xR$};
\node (NE) at (\L,\L){$R/x^2 R$};
\node (NEE) at (3*\L,\L) {$\cdots$};
\node (SW) at (-\L,0){$F_* (R/x^p R)$};
\node (SE) at (\L,0){$F_* (R/x^{2p}R)$};
\node (SEE) at (3*\L,0) {$\cdots$};

\draw[->] (NW) edge node[above]{$\cdot x$} (NE);
\draw[->] (NE) edge node[above]{$\cdot x$} (NEE);

\draw[->] (SW) edge node[above]{$\cdot x^p$}(SE);
\draw[->] (SE) edge node[above]{$\cdot x^p$}(SEE);

\draw[->] (NW) edge  node[right]{$\rho_F$} (SW);
\draw[->] (NE) edge node[right]{$\rho_F$} (SE);

\end{tikzpicture}
\end{equation}

Taking direct limits on the rows of Diagram~\eqref{eq:MainDiagLim} and applying $H^i_{\fm}(-)$, we get two directed systems $\{ H_\fm^i(R/x^\ell R)\}_{\ell > 0}$ and $\{ H_\fm^i(R/x^{p\ell} R)\}_{\ell > 0}$ with Frobenius linear maps 
$$
\rho_{\ell,i} \colon H_\fm^i(R/x^\ell R) \to F_* H_\fm^i(R/x^{p\ell}R)
$$ 
which are injective for each $\ell > 0$ by Theorem \ref{theorem1}. Thus the collection of injective Frobenius linear maps $\rho_{\ell,i} \colon H^i_{\fm}(R/x^{\ell}R) \to F_*H^i_{\fm}(R/x^{p\ell}R)$ induce an injective Frobenius linear map 
$$
\rho_1=\varinjlim_{\ell} \rho_{\ell,i}:\varinjlim_{\ell} H^i_{\fm}(R/x^{\ell}R) \to F_* \varinjlim_{\ell} H^i_{\fm}(R/x^{p\ell}R),
$$ 
since $F_*$ commutes with $\varinjlim$ by Lemma \ref{lem:Finjlim}. The module $H_{(x)}^1(R)$ has a natural Frobenius action induced from the Frobenius on $R$ which in turn induces a Frobenius action $\rho_2$ on $H^i_{\fm}(H^1_{(x)}(R))$. Let $\rho_3$ denote the natural Frobenius action on $H^{i+1}_{\fm}(R)$.

It suffices to show that the following diagram commutes for each $i \geq 0$.
\begin{equation}
\label{diag2}
\begin{tikzpicture}[baseline=(current bounding box.center)]
\pgfmathsetmacro{\L}{2}
\node (NW) at (-\L,\L) {$\varinjlim_\ell H_\fm^i(R/x^\ell R)$};
\node (NE) at (\L,\L){$H_\fm^i(H_{(x)}^1(R))$};
\node (NEE) at (3*\L,\L) {$H^{i+1}_{\fm}(R)$};
\node (SW) at (-\L,0){$\varinjlim_\ell F_*H^i_{\fm}(R/x^{p\ell} R)$};
\node (SE) at (\L,0){$F_* H^i_{\fm}(H^1_{(x)}(R))$};
\node (SEE) at (3*\L,0) {$F_* H^{i+1}_{\fm}(R)$};

\draw[->] (NW) edge node[above]{$\alpha_1$} (NE);
\draw[->] (NE) edge node[above]{$\beta_1$} (NEE);

\draw[->] (SW) edge node[above]{$F_* \alpha_2$}(SE);
\draw[->] (SE) edge node[above]{$F_* \beta_2$}(SEE);

\draw[->] (NW) edge  node[right]{$\rho_1$} (SW);
\draw[->] (NE) edge node[right]{$\rho_2$} (SE);
\draw[->] (NEE) edge node[right]{$\rho_3$} (SEE);

\end{tikzpicture}
\end{equation}
where $\alpha_1$ and $F_*\alpha_2$ are the isomorphisms coming from Lemma \ref{lem:isos}, and $\beta_1$ and $F_*\beta_2$ are the isomorphisms coming from Lemma \ref{lem:SS}. Since $\rho_1$ is injective, it follows from this commutativity that $\rho_3$ is injective. We show Diagram~\eqref{diag2} commutes by splitting it into two commuting squares.

To show the first square in Diagram~\eqref{diag2} commutes, note that this square is just applying $H^i_\fm(-)$ to the following square, where the vertical Frobenius linear maps are those induced by the natural Frobenius on $R$. 

\begin{center}
\begin{tikzpicture}
\pgfmathsetmacro{\L}{1.5}
\node (NW) at (-\L,\L) {$\varinjlim_\ell R/x^\ell R$};
\node (NE) at (\L,\L){$H_{(x)}^1(R)$};
\node (SW) at (-\L,0){$\varinjlim_\ell F_* (R/x^{p\ell} R)$};
\node (SE) at (\L,0){$F_*H_{(x)}^1(R)$};

\draw[->] (NW) edge node[above]{$\cong$} (NE);
\draw[->] (NW) edge  (SW);
\draw[->] (NE) edge (SE);
\draw[->] (SW) edge node[above]{$\cong$}(SE);

\end{tikzpicture}
\end{center}

The second square in Diagram~\eqref{diag2} commutes since $\RGamma_\fm \circ \RGamma_{(x)} \cong \RGamma_\fm$ in the derived category by \cite[Cor. 10.8.3]{Wei94} and we are simply applying each functor to the natural Frobenius action $\rho_F \colon R \to F_* R$. That is to say, $\RGamma_\fm ( \RGamma_{(x)} ( \rho_F)) = \RGamma_\fm(\rho_F)$. 
\end{proof}

\subsection{Deforming surjectivity of Frobenius linear maps}

Clearly Frobenius linear maps are not generally surjective. However, often it is surjective ``up to Frobenius". To make this clear, we start with a simple example.

\begin{example}
Let $k$ be a perfect field of characteristic $p > 0$. The natural Frobenius action $k[x] \to F_*k[x]$ is $k[x]$-linear and has image $F_* k[x^p]$ in $F_* k[x]$. So it is not surjective. However, it is surjective up to $F_* k[x]$-span in the sense that the singleton set $\{F_*1\}$ forms a $F_*k[x]$-basis of $F_*k[x]$.
\end{example}

\begin{definition} \label{dff:equptoF_*Rspan} Let $R$ be a ring of characteristic $p$ and $M$ and $N$ be $R$-modules. Call an $e$-th iterated Frobenius linear map $\rho:M \to F^e_*N$ \textit{surjective up to $F^e_*R$-span}, when the $F^e_*R$-span of $\im(\rho)$ is equal to $F^e_*N$. 
\end{definition}

The condition in Definition~\ref{dff:equptoF_*Rspan} is equivalent to having a set $\{a_i\}_{i \in \Lambda}$ of  generators for $M$ for which $F_*^eN$ is the $F^e_*R$-submodule of $F^e_*N$ spanned by $\{\rho(a_i)\}_{i \in \Lambda}$. This section investigates how this property deforms.

We leave it to the reader to check for a directed system of $R$-modules $\{M_i\}_{i \in I}$  and Frobenius actions $\phi_i \colon M_i \to F_*M_i$ for each $i \in I$ with each $\phi_i$ surjective up to $F_*R$-span, the natural induced map $\phi=\varinjlim_i \phi_i \colon \varinjlim_i M_i \to \varinjlim_i F_* M_i$ is also surjective up to $F_*R$-span.

\begin{lemma}
\label{final}
Let $R$ be a commutative ring of characteristic $p>0$ and assume that
$$
\begin{CD}
L @>\alpha_1>> M @>\alpha_2>> N \\
@V\rho_1VV @V\rho_2VV @V\rho_3VV \\
F_*L' @>F_*\alpha'_1>> F_*M' @>F_*\alpha'_2>> F_*N'  \\
\end{CD}
$$
is a commutative diagram where $L,M,N,L',M',$ and $N'$ are $R$-modules such that the top row is $R$-linear and exact and the bottom row is $F_*R$-linear and exact, and such that each $\rho_i$ is a Frobenius linear map for $i=1,2,3$. If $\rho_1$ and $\rho_3$ are surjective up to $F_*R$-span and $\alpha_2$ is surjective, then  $\rho_2$ is also surjective up to $F_*R$-span..
\end{lemma}

\begin{proof}
Choose sets of generators of $R$-modules $L$, $M$, and $N$, say $\{x_i\}$, $\{y_j\}$, and $\{z_k\}$ respectively. Without loss of generality, we may assume $\{\alpha_1(x_i)\} \subseteq \{y_j\}$ and $\alpha_2(\{y_j\} \backslash \{\alpha_1(x_i)\})=\{z_k\}$. It suffices to show each element of $F_*M'$ can be presented as an $F_*R$-linear combination of $\{\rho_2(y_j)\}$. Pick $F_* m \in F_*M'$ and consider $F_*\alpha'_2( F_*m ) \in F_*N'$. By hypothesis, we can write
\begin{equation}
\label{eq:surjRSpan1}F_*\alpha'_2(F_*m )=\sum_i F_*c_i \rho_3(z_i) 
\end{equation}
with $F_*c_i \in F_*R$. Now let $y'_i \in M$ be the inverse image of each $z_i \in N$ appearing in the equation $(\ref{eq:surjRSpan1})$. By our set up, we have $y'_i \in \{y_j\}$. By commutativity of the diagram, we also have $$F_* m -\sum_i F_*c_i\rho_2(y'_i) \in \ker F_*\alpha'_2.$$ Since the bottom row is exact, one has $$F_* m -\sum_i F_*c_i\rho_2(y'_i) = \sum_j F_*a_j F_*\alpha'_1(\rho_1(x_j))$$ for some $F_*a_j \in F_*R$ and thus
$$
F_*m =\sum_j F_*a_j F_*\alpha'_1(\rho_1(x_j))+\sum_i F_*c_i\rho_2(y'_i) = \sum_j F_*a_j \big(F_*\alpha'_1(\rho_1(x_j))\big)+\sum_i F_*c_i\rho_2(y'_i),
$$
which proves the lemma, since each $F_*\alpha'_1(\rho_1(x_j)) \in \{\rho_2(y_i)\}$.
\end{proof}

As a corollary, we obtain the following.

\begin{theorem}
\label{thm:surjRspan}
Let $(R,\fm)$ be a local ring of characteristic $p>0$ with $x \in \fm$ a surjective element. If the Frobenius linear map $H_\fm^i(R/xR) \to F_* H_\fm^i(R/x^pR)$ is surjective up to $F_*R$-span for all $i \ge 0$, then the Frobenius action 
$$
H^i_{\fm}(R) \to F_*H^i_{\fm}(R)
$$ 
is also surjective up to $F_*R$-span for all $i \ge 0$.
\end{theorem}

\begin{proof}
We utilize the notation and setup from the proof of Theorem~\ref{theorem1}. We start by showing that $\rho_1 := \varinjlim_\ell \rho_{\ell,i}$ is surjective up to $F_*R$-span and to do so, it suffices to check that each $\rho_{\ell, i}$ is surjective up to $F_*R$-span. Proceed by induction on $\ell > 0$ (defined in the proof of Theorem~\ref{theorem1}), where the base case, i.e., that  $\rho_{1,i}$ is surjective up to $F_*R$-span, is guaranteed by hypothesis. We assume $\rho_{\ell-1,i}$ is surjective up to $F_*R$-span. Note that Diagram~\eqref{eq:diag3} of Theorem \ref{theorem1} has exact rows and by Lemma \ref{final},  $\rho_{\ell,i}$ is surjective up to $F_*R$-span for all $\ell > 0$. 

Now proceed as in the proof of Theorem~\ref{thm:MainTheorem}. Here $\rho_1=\varinjlim_\ell \rho_{\ell,i}$ is surjective up to $F_*R$-span and  $\beta_1 \circ \alpha_1$ and $F_*\beta_2 \circ F_*\alpha_2$ are isomorphisms. From Diagram~\eqref{diag2}, we see that $\rho_3$ is surjective up to $F_*R$-span as well. Putting this together we have shown the Frobenius action 
$$
H^{i+1}_{\fm}(R) \to F_*H^{i+1}_{\fm}(R)
$$
is surjective up to $F_*R$-span for $i \ge 0$, as desired.
\end{proof}

\section{Applications}

Utilizing Theorem \ref{thm:MainTheorem}, we now describe two conditions for when $F$-injectivity deforms. One is a finite length condition on local cohomology modules, the other is $F$-purity. Both can be stated in terms of Frobenius actions on local cohomology using the notion of anti-nilpotent modules.

\subsection{Finite Length Cohomology}

The first case that we can apply our main theorem to is one utilizing a finiteness condition on local cohomology modules.

\begin{definition}
For a local ring $(R,\fm)$, we say an $R$-module $M$ has \textit{finite local cohomology} (FLC) provided the local cohomology module $H^i_{\fm}(M)$ has finite length for all $i \leq \dim M-1$. 
\end{definition}

\begin{remark} 
Sometimes when a local ring $R$ has FLC it is called a \textit{generalized Cohen-Macaulay ring}. When $R$ has a dualizing complex, this means exactly that the non-CM locus of $R$ is isolated \cite{Sch75}. 
\end{remark}

In the setting of a local ring $(R,\fm)$ with $x \in \fm$ a regular element, we are most concerned with the $R$-modules $R$ and $R/x^\ell R$; i.e., an infinitesimal neighborhood of the special fiber. We now show that FLC extends to such neighborhoods when imposed on the special fiber.

\begin{lemma}
\label{lemma3}
Let $(R,\fm,k)$ be a local ring with $x \in R$ a regular element such that $\fm^s \cdot H_\fm^i(R/xR) = 0$ for some $s \ge 0$. For each $\ell > 0$, we have 
$$
\fm^{s\ell} \cdot H_\fm^i(R/x^\ell R) = 0.
$$
In particular, if $R/xR$ has FLC, so does $R/x^\ell R$. 
\end{lemma}

\begin{proof}
We show this by induction on $\ell$. If $\ell = 1$, then this is just the hypothesis. Assume $\ell > 1$ and $\fm^{sj} \cdot H_\fm^i(R/x^j R) = 0$ for all $j < \ell$. The short exact sequence 
$$
0 \to R/x^{\ell -1}R \xrightarrow{x} R/x^{\ell} R \to R/xR \to 0,
$$ 
induces a long exact sequence in local cohomology. We only need the portion
$$
H_\fm^i(R/x^{\ell -1}R) \xrightarrow{\alpha} H_\fm^i(R/x^\ell R) \xrightarrow{\beta} H_\fm^i(R/xR),
$$ 
which is an exact sequence of $R$-modules. Take an element $\eta \in H_\fm^i(R/x^\ell R)$ and $c \in \fm^s$. One has $\beta(c \eta)=c \beta(\eta) = 0$, which implies that $c \eta$ has a preimage $\theta \in H_\fm^i(R/x^{\ell -1} R)$ along $\alpha$. By induction, we have $m \cdot \theta = 0$ for any $m \in \fm^{s(\ell -1)}$. Therefore, $\alpha(m \cdot  \theta ) = 0$ and $m \cdot c \eta = (mc) \cdot \eta = 0$. Since $c$ and $m$ were chosen arbitrarily, we have that $\fm^{s\ell} \cdot H_\fm^i(R/x^\ell R) = 0$. 
\end{proof}

\begin{remark} 
We note that there was no restriction on the characteristic on rings in Lemma \ref{lemma3}. 
\end{remark}

An easy consequence of the FLC property is a result on surjective maps of local cohomology.

\begin{lemma}
\label{lem:surj}
Let $(R,\fm,k)$ be a local ring of characteristic $p > 0$ with perfect residue field $k$ and $x \in \fm$ a regular element. Assume that $R/xR$ is $F$-injective and FLC. For each $\ell > 0$, the surjection $R/x^\ell R \to R/xR$ induces a surjection 
$$
H_\fm^i(R/x^\ell R) \to H_\fm^i(R/xR)
$$ 
for each $0 \leq i \leq \dim R - 2$. 
\end{lemma}

\begin{proof}
By Lemma \ref{lemma2}, since $R/xR$ has FLC and is F-injective with perfect residue field, for $i$ in the interval $[0,\dim R - 2]$ the $e$-th iterated Frobenius action 
$$
H_\fm^i(R/xR) \to F^e_* H_\fm^i(R/xR)
$$ 
induced by Frobenius on $R/xR$ is surjective. For $\ell>0$, choose $e \gg 0$ so that the surjection $R/x^{p^e}R \twoheadrightarrow R/xR$ factors as $R/x^{p^e}R \twoheadrightarrow R/x^{\ell}R \twoheadrightarrow R/xR$. This induces a composition of maps:
$$
H_\fm^i(R/xR) \to F^e_*H_\fm^i(R/x^{p^e} R) \to F^e_*H_\fm^i(R/ x^\ell R) \to F^e_*H_\fm^i(R/xR).
$$ 
The composition is surjective and so $H_\fm^i(R/x^\ell R) \to H_\fm^i(R/xR)$ must be. 
\end{proof}

\begin{remark}
The assumption that the residue field of $R$ is perfect is necessary in the proof of Lemma~\ref{lem:surj}. If $R$ is F-injective and contains a non-perfect field $K$, it is not necessarily true that $R \otimes_K K^{1/p}$ is $F$-injective. For example, set $K:=\mathbf{F}_p(x)$. Note that $R:=K[t]/(t^p-x)$ is a field, however, $R \otimes_K K^{1/p} \cong K^{1/p}[t]/(t-x^{1/p})^p$ is not reduced and hence not $F$-injective. 
\end{remark}

\begin{corollary} 
\label{cor:FLCFINJ}
Let $(R,\fm,k)$ be a local ring of characteristic $p$ with perfect residue field and $x \in \fm$ a regular element. If $R/xR$ has FLC and is $F$-injective, then $R$ is $F$-injective. 
\end{corollary}

\begin{proof}
We utilize the same notation as Theorem~\ref{theorem1} and Theorem~\ref{thm:MainTheorem}. Applying Lemma~\ref{lem:surj}, we see that $H_\fm^i(R/x^\ell R) \to H_\fm^i(R/xR)$ is surjective for all $i \in [0,\dim R -2]$ and $\ell > 0$. Now, following the proof of Lemma~\ref{cor:inj},  $H_\fm^i(R/x^\ell R) \to H_\fm^i(R/x^{\ell + 1} R)$ is injective for all $\ell > 0$ and $i \in [0,\dim R -1]$. This suffices in the proof of Theorem~\ref{theorem1} to conclude that $\rho_{i,\ell} \colon H_\fm^i(R/x^\ell R) \to F_* H_\fm^i(R/x^{p \ell}R)$ is injective for $i \in [0,\dim R -1]$ and all $ \ell > 0$. Finally, this is sufficient to apply the proof of Theorem~\ref{thm:MainTheorem} to conclude that $R$ is $F$-injective. 
\end{proof}

Immediately this shows that potential counterexamples to the deformation of $F$-injectivity in nice geometric settings must have dimension at least $4$.

\begin{corollary}
\label{cor:FinjGeomlowdim}
If $(R,\fm,k)$ is a complete local ring of characteristic $p > 0$ with perfect residue field and dimension at most 4 and $x \in \fm$ is a regular element with $R/xR$ normal and $F$-injective, then $R$ is $F$-injective. 
\end{corollary}

\begin{proof}
Since $R/xR$ is a local normal domain and $x \in \fm$ is a regular element, $R$ is also normal by (\cite{EGA} 5.12.7). In particular, $R$ is a domain and equidimensional. Since $\dim R \leq 4$, one has $\dim R/xR \leq 3$. By normality of $R/xR$, it satisfies Serre's condition $S_2$, therefore the non-CM locus is isolated, hence $R/xR$ has FLC and by Corollary \ref{cor:FLCFINJ}, $R$ must be $F$-injective. 
\end{proof}

\begin{example}
For any ring $A$ which is not Cohen Macaulay, has FLC, and $F$-split, the ring $R := A[[x]]$ does not have FLC. However, $R/xR$ is $F$-injective and has FLC. In particular, consider 
$$
A=\mathbf{F}_p[[a,b,c,d]]/(a,b) \cap (c,d).
$$ 
Note that $A$ has FLC and is even Buchsbaum; see \cite{GoOg}. It is also not Cohen-Macaulay, but is $F$-split by Fedder's criterion \cite{Fed}. Thus $A[[x]]$ is $F$-injective and the non-CM locus of $R$ is defined by the non-maximal ideal $\fn R$ where $\fn$ is the maximal ideal $\fn$ of $A$.
\end{example}

\subsection{$F$-splitting and $F$-injectivity}

The second application concerns $F$-purity. We utilize work of L. Ma \cite{Ma} building on work by Enescu and Hochster \cite{EH}. The language used in \cite{EH} is in terms of $R\{F\}$-modules which are modules over a ring $R$ with a specified Frobenius action. For such a module $M$ with a distinguished Frobenius action $\rho \colon M \to F_* M$, a submodule $N \subset M$ is called \textit{$F$-compatible}, provided that $\rho(N) \subseteq F_*N$. Ma showed that $F$-split local rings  have local cohomology modules, which when equipped with the natural Frobenius action, satisfy an interesting condition, originally introduced in \cite{EH}.

\begin{definition}
$($\cite[Def. 4.6]{EH}$)$
Let $(R,\fm)$ be a local ring. An $R$-module $M$ with a Frobenius action $\rho$ is called \textit{anti-nilpotent}, provided for any  $F$-compatible submodule $N$ (i.e., $\rho(N) \subseteq F_* N$), the induced action of $\rho$ on $M/N$ is injective. 
\end{definition}

\begin{theorem}
\label{thm:FsplitFinj} 
Let $(R,\fm,k)$ be a local ring of characteristic $p > 0$ and $x \in \fm$ a regular element. If $H_\fm^i(R/xR)$ is anti-nilpotent for all $i \geq 0$ then $x$ is a surjective element. 
\end{theorem}

\begin{proof}
By definition, we need to check that the map $H_\fm^i(R/x^\ell R) \to H_\fm^i(R/xR)$, which is induced by the surjection $R/x^\ell R \to R/xR$, is surjective. Denote its cokernel by $C$. It suffices to show that $C = 0$. Consider the exact sequence $H_\fm^i(R/x^\ell R) \to H_\fm^i(R/xR) \to C \to 0$. Denote by $\rho_{\ell, i}^e \colon H_\fm^i(R/x^\ell R) \to F_*^e H_\fm^i(R/x^\ell R)$ the Frobenius linear map induced naturally by the Frobenius on $R$ composed with the natural surjection. 


The map $\rho_{1,i}^e$ induces a Frobenius linear map $C \to F_*^e C$ and denote this by $\rho_C^e$. These Frobenius linear maps fit together to give a commutative diagram with exact rows since $F_*^e$ is exact for all $e$. 

\begin{center}
\begin{tikzpicture}

\node (11) at (-4,4) {$H_\fm^i(R/x^\ell R)$};
\node (12) at (0,4){$H_\fm^i(R/xR)$};
\node (13) at (3,4){$C$};
\node (14) at (5,4){$0$};

\node (21) at (-4,2) {$F_*^e H_\fm^i(R/x^{\ell} R)$};
\node (22) at (0,2){$F_*^e H_\fm^i(R/x R)$};
\node (23) at (3,2){$F_*^e C$};
\node (24) at (5,2){$0$};

\draw[->] (11) edge (12);
\draw[->] (12) edge (13);
\draw[->] (13) edge (14);

\draw[->] (21) edge (22);
\draw[->] (22) edge (23);
\draw[->] (23) edge (24);

\draw[->] (11) edge node[left]{$\rho_{\ell,i}^e$} (21);
\draw[->] (12) edge node[left]{$\rho_{1,i}^e$} (22);
\draw[->] (13) edge node[left]{$\rho_C^e$}(23);
\draw[->] (14) edge (24);

\end{tikzpicture}
\end{center}

The image of $H_\fm^i(R/x^\ell R)$ in $H_\fm^i(R/xR)$ is certainly $F$-compatible. Since we assume $H_\fm^i(R/xR)$ is anti-nilpotent, the Frobenius action $\rho_C^e$ on $C$ is injective. Note also that when $e \gg 0$, the map $\rho_{1,i}^e$ factors as $$H_\fm^i(R/xR) \to F_*^e H_\fm^i(R/x^{p^e} R) \to F_*^e H_\fm^i(R/x^\ell R) \to F_*^e H_\fm^i(R/xR).$$ So we may define the map $\varphi$ making the following diagram commute:

\begin{equation}
\label{diag:Fpure}
\begin{tikzpicture}[baseline=(current bounding box.center)]

\node (11) at (-4,4) {$H_\fm^i(R/x^\ell R)$};
\node (12) at (0,4){$H_\fm^i(R/xR)$};
\node (13) at (3,4){$C$};
\node (14) at (5,4){$0$};

\node (21) at (-4,2) {$F^e_* H_\fm^i(R/x^{\ell} R)$};
\node (22) at (0,2){$F^e_* H_\fm^i(R/x R)$};
\node (23) at (3,2){$F^e_* C$};
\node (24) at (5,2){$0$};

\draw[->] (11) edge (12);
\draw[->] (12) edge (13);
\draw[->] (13) edge (14);

\draw[->] (21) edge (22);
\draw[->] (22) edge (23);
\draw[->] (23) edge (24);

\draw[->] (11) edge node[left]{$\rho_{\ell,i}^e$} (21);
\draw[->] (12) edge node[left]{$\rho_{1,i}^e$} (22);
\draw[->] (13) edge node[left]{$\rho_C^e$}(23);
\draw[->] (14) edge (24);

\draw[->] (12) edge node[above]{$\varphi$}(21);

\end{tikzpicture}
\end{equation}

We show that $C=0$ by utilizing a diagram chase on \eqref{diag:Fpure}. Let $z \in C$. As such, it has a preimage $z' \in H_\fm^i(R/xR)$. By commutativity of the diagram, it follows that $\rho_{1,i}^e( z')$ has preimage $z''=\varphi(z')$. As the bottom row is exact, $z''$ maps to $\rho^e_C(z)$ which is zero. However $\rho^e_C$ was shown to be injective, this implies that $z=0$ and therefore $C=0$, as desired.  
\end{proof}

\begin{remark}
The proof of Theorem~\ref{thm:FsplitFinj} can be modified to show that when the natural Frobenius linear map $H_\fm^i(R/xR) \to F_*H_\fm^i(R/xR)$ is surjective up to $F_*R$-span, for each $\ell > 0$ the map $H_\fm^i(R/x^\ell R) \to H_\fm^i(R/xR)$ is surjective. 
\end{remark}

\begin{corollary}
\label{cor:FsplitFinj}
Let $(R,\fm,k)$ be a local ring of characteristic $p > 0$ and $x \in \fm$ a regular element. If $R/xR$ is $F$-split then $R$ is $F$-injective. 
\end{corollary}
\begin{proof}
Since $R/xR$ is $F$-split, the module $H_\fm^i(R/xR)$ is anti-nilpotent for all $i \geq 0$ by \cite[Thm. 3.7]{Ma} and so Theorem~\ref{thm:FsplitFinj} gives that $x$ is a surjective element. The rest follows by Theorem~\ref{thm:MainTheorem}. 
\end{proof}

\begin{remark}
We note when the residue field is perfect and $R/xR$ is $F$-injective has FLC then $H_\fm^i(R/xR)$ is anti-nilpotent for all $i < \dim R/xR$  as $H_\fm^i(R/xR)$ is a finite dimensional $k$-vector space and so Frobenius acts injectively. Thus one may use Theorem~\ref{thm:FsplitFinj} to replace the role of Lemma~\ref{lem:surj} in the proof of Corollary~\ref{cor:FLCFINJ}. 
\end{remark}

\begin{remark}
Under an $F$-finite assumption, Theorem \ref{thm:FsplitFinj} says that $F$-purity deforms to $F$-injectivity. Enescu obtained some results on this finiteness property on local cohomology modules of finite length \cite[Thm. 7.14]{Ene}. 
\end{remark}

\begin{example}
A particularly well-known example where $F$-purity fails to deform was introduced by Fedder \cite{Fed} see also \cite[Example 3.2]{Singha}. In particular, the ring $$R:=\mathbf{F}_p[[X,Y,Z,W]]/(XY,XW,W(Y-Z^2))$$ is not $F$-pure, but $R/ZR=\mathbf{F}_p[[X,Y,W]]/(XY,XW,WY)$ is known to be $F$-pure \cite[Prop. 5.38]{HR76}. This also means that our main result serves as a way for checking $F$-injectivity by taking specialization, $i.e.,$ checking that $R/ZR$ is $F$-pure.
\end{example}

\appendix\section{$F$-injectivity and depth}

\begin{center} by Karl Schwede and Anurag K. Singh \end{center}

\

Our goal here is to prove a prime characteristic analog of a result of Koll\'ar and Kov\'acs, \cite[Theorem~7.12]{Kollar-Kovacs}: if $X\to B$ is a flat family with Du~Bois fibers, such that the generic fiber is Cohen-Macaulay (respectively $S_k$), then all fibers of the map~$X\to B$ are Cohen-Macaulay (respectively $S_k$). The prime characteristic version of this is Theorem~\ref{theorem:finj:main} below. As applications of this theorem, we extend a result of Fedder and Watanabe~\cite[Proposition~2.13]{FedderWatanabe} to the case where $R$ is not a priori assumed to be Cohen-Macaulay, see Corollary~\ref{cor:finj:frat}, and also obtain a new result on the deformation of $F$-injectivity, Corollary~\ref{cor:finj:def}.

We begin with some preliminary observations:

\begin{lemma}
\label{lemma:finj:def}
Let $(R,\fm)$ be a local ring; set $d$ to be the depth of $R$. Suppose there exists a regular element $f$ in $R$ such that the Frobenius action on $H^{d-1}_\fm(R/fR)$ is injective. Then the map
\[
\CD
H^d_\fm(R)@>f^{p-1}F>>H^d_\fm(R)\,
\endCD
\]
is injective; in particular, the Frobenius action on $H^d_\fm(R)$ is injective.
\end{lemma}

\begin{proof}
Consider the commutative diagram with exact rows:
\[
\CD
0@>>>R@>f>>R@>>>R/fR@>>>0\phantom{\,.}\\
@. @VVf^{p-1}FV @VVFV @VVFV \\
0@>>>R@>f>>R@>>>R/fR@>>>0\,.
\endCD
\]
Since $R/fR$ has depth $d-1$, applying the functor $H^\bullet_\fm(\ \ )$ yields the diagram
\[
\CD
0@>>>H^{d-1}_\fm(R/fR)@>>>H^d_\fm(R)@>f>>H^d_\fm(R)@>>>H^d_\fm(R/fR)\phantom{\,.}\\
@. @VVFV @VVf^{p-1}FV @VVFV @VVFV\\
0@>>>H^{d-1}_\fm(R/fR)@>>>H^d_\fm(R)@>f>>H^d_\fm(R)@>>>H^d_\fm(R/fR)\,.
\endCD
\]
The map $f^{p-1}F$ is injective if and only if it is injective when restricted to the socle of $H^d_\fm(R)$. The socle is annihilated by $f$, and thus lies in the image of $H^{d-1}_\fm(R/fR)$. But the Frobenius action on $H^{d-1}_\fm(R/fR)$ is injective by assumption.
\end{proof}

The next lemma is the main ingredient in the proof of Theorem~\ref{theorem:finj:main}. For a local ring $(R,\fm)$, we use $\PSpec R$ to denote the \emph{punctured spectrum} of~$R$, i.e., the set~$\Spec R\smallsetminus\{\fm\}$. The $F$-finite hypothesis in the sequel ensures the existence of a dualizing complex by Gabber, \cite[Remark~13.6]{Gabber}. By Kawasaki \cite[Cor 1.4]{Kaw02}, local rings possessing dualizing complexes are precisely those that are homomorphic images of Gorenstein local rings.

\begin{lemma}
\label{lemma:finj:isolated}
Let $(R,\fm)$ be an $F$-finite local ring. Suppose there exists a regular element $f$ in $R$ such that $R/fR$ is $F$-injective.

If $R_\fp$ satisfies the Serre condition $S_k$\, for each $\fp$ in $\PSpec R$, then $R$ satisfies $S_k$.
\end{lemma}

\begin{proof}
Let $d$ be the depth of $R$. If $R$ does not satisfy $S_k$, then $d<k$.

The module $H^d_\fm(R)$ is nonzero, but has finite length since $R_\fp$ satisfies $S_k$ for each prime ideal $\fp$ in $\PSpec R$. We claim that $\fm\,H^d_\fm(R)=0$. Because it has finite length, the module $H^d_\fm(R)$ is annihilated by $\fm^q$ for some $q=p^e$. For each $x\in\fm$ and $\eta\in H^d_\fm(R)$, it follows that $x^qF^e(\eta)=0$. But the Frobenius action on $H^d_\fm(R)$ is injective by Lemma~\ref{lemma:finj:def}, so $x\eta=0$, which proves the claim.

But then $f^{p-1}FH^d_\fm(R)=0$. Since $f^{p-1}F\colon H^d_\fm(R)\to H^d_\fm(R)$ is injective by Lemma~\ref{lemma:finj:def}, we must have $H^d_\fm(R)=0$, which is a contradiction.
\end{proof}

\begin{theorem}
\label{theorem:finj:main}
Let $R$ be an $F$-finite local ring. Suppose there exists a regular element $f$ in $R$ such that $R/fR$ is $F$-injective.

If the localization $R_f = R[f^{-1}]$ satisfies the Serre condition $S_k$\, for a positive integer $k$, then $R$ satisfies condition~$S_k$. In particular, if $R_f$ is Cohen-Macaulay, then $R$ is Cohen-Macaulay.
\end{theorem}

\begin{proof}
If not, take a prime $\fq$ that is minimal with respect to the property that $R_\fq$ does not satisfy $S_k$. As $R_f$ is $S_k$ by assumption, it follows that $f\in\fq$. Since it is a localization of an $F$-injective ring, the ring $(R/fR)_\fq=R_\fq/fR_\fq$ is $F$-injective, see, for example, \cite[Proposition~4.3]{Sch}. But $(R_\fq)_\fp$ satisfies condition $S_k$ for each prime ideal $\fp$ in $\PSpec R_\fq$, so $R_\fq$ satisfies $S_k$ by Lemma~\ref{lemma:finj:isolated}. This is a contradiction.
\end{proof}

The following corollary was proved as \cite[Proposition~2.13]{FedderWatanabe} under the additional hypothesis that $R$ is Cohen-Macaulay:

\begin{corollary}
\label{cor:finj:frat}
Let $R$ be an $F$-finite local ring. Suppose there exists a regular element $f$ in $R$ such that $R/fR$ is $F$-injective. If $R_f$ is $F$-rational, then $R$ is $F$-rational.
\end{corollary}

\begin{proof}
Theorem~\ref{theorem:finj:main} implies that $R$ is Cohen-Macaulay. But then $R$ is $F$-rational by \cite[Proposition~2.13]{FedderWatanabe}; Fedder and Watanabe require $R_f$ to be regular in the statement of the proposition, but their proof works verbatim if some power of $f$ is a parameter test element, and this is indeed the case by \cite[Theorem~1.13]{Velez}.
\end{proof}

Fedder~\cite[Theorem~3.4\,(1)]{Fed} proved that $F$-injectivity deforms in the case of Cohen-Macaulay rings; we extend this as follows:

\begin{corollary}
\label{cor:finj:def}
Let $R$ be an $F$-finite local ring. If $f\in R$ is a regular element such that $R/fR$ is $F$-injective, and $R_f$ is Cohen-Macaulay, then $R$ is $F$-injective.
\end{corollary}

\begin{proof}
Theorem~\ref{theorem:finj:main} implies that the ring $R$ is Cohen-Macaulay; we may then use~\cite[Theorem 3.4.1]{Fed}.
\end{proof}

\end{document}